\documentclass[reqno,12pt]{amsart}
\usepackage{amsfonts}
\usepackage{bbm}
\usepackage{} 
\numberwithin{equation}{section}
\usepackage{color}

\usepackage{indentfirst}
\usepackage{color}
\usepackage{amssymb}
\usepackage{mathrsfs}
\usepackage{xy}
\xyoption{all}
\def\Ext{\mbox{\rm Ext}\,} \def\Hom{\mbox{\rm Hom}} \def\dim{\mbox{\rm dim}\,} 
\def\lr#1{\langle #1\rangle} \def\fin{\hfill$\square$}   \def\mod{\mbox{\rm \textbf{mod}}\,}
   \def\im{\mbox{\rm Im}\,} 
\def\End{\mbox{\rm End}\,}
 \def\gl.{\mbox{\rm gl.}\,}\def\D{\mbox{\rm D}\,}
\def\ind{\mbox{\rm ind}\,}\def\P{\mathcal {P}}\def\R{\mathcal {R}}\def\I{\mathcal {I}}
\def\Aut{\mbox{\rm Aut}\,}

\theoremstyle{plain} 
\newtheorem{theorem}{\bf Theorem}[section]
\newtheorem{lemma}[theorem]{\bf Lemma}
\newtheorem{corollary}[theorem]{\bf Corollary}
\newtheorem{proposition}[theorem]{\bf Proposition}

\theoremstyle{definition}
\newtheorem{definition}[theorem]{\bf Definition}
\newtheorem{remark}[theorem]{\bf Remark}
\newtheorem{example}[theorem]{\bf Example}

\newcommand{\bt}{\begin{theorem}}
\newcommand{\et}{\end{theorem}}
\newcommand{\bl}{\begin{lemma}}
\newcommand{\el}{\end{lemma}}
\newcommand{\bd}{\begin{definition}}
\newcommand{\ed}{\end{definition}}
\newcommand{\bc}{\begin{corollary}}
\newcommand{\ec}{\end{corollary}}
\newcommand{\bp}{\begin{proof}}
\newcommand{\ep}{\end{proof}}
\newcommand{\bx}{\begin{example}}
\newcommand{\ex}{\end{example}}
\newcommand{\br}{\begin{remark}}
\newcommand{\er}{\end{remark}}
\newcommand{\be}{\begin{equation}}
\newcommand{\ee}{\end{equation}}
\newcommand{\ba}{\begin{align}}
\newcommand{\ea}{\end{align}}
\newcommand{\bn}{\begin{enumerate}}
\newcommand{\en}{\end{enumerate}}
\newcommand{\bcs}{\begin{cases}}
\newcommand{\ecs}{\end{cases}}

\makeatletter
\renewcommand{\section}{\@startsection{section}{1}{0mm}
  {-\baselineskip}{0.5\baselineskip}{\bf\leftline}}
\makeatother

\begin{document}

\title[On derived Hall numbers for tame quivers]{On derived Hall numbers for tame quivers}
\author{Shiquan Ruan and Haicheng Zhang}
\address{Yau Mathematical Sciences Center, Tsinghua University,
Beijing 100084,  China.} \email{sqruan@math.tsinghua.edu.cn,\;\; zhanghai14@mails.tsinghua.edu.cn}


\begin{abstract}
In the present paper we study the derived Hall algebra for the bounded derived category of the nilpotent representations of a tame quiver over a finite field. We show that for any three given objects in the bounded derived category, the associated derived Hall numbers are given by a rational function in the cardinalities of ground fields.
\end{abstract}

\maketitle

%%%%%%%%%%%%%%%%%%%%%%%%%%%%%%%%%%%%%%%%%%%%%%%%%%%%%%%%%%%%%%%%%%%%%%%%%%%%%%%%%%%%%%%%%%%%%%%%%%%%%%%%%%%%%%%%%%%%%%%%%%%%%%%%%%%%%%%%%%%%%%%%%%%%%%%%
\section{Introduction}

In 1990, Ringel \cite{R90a} introduced the Hall algebra $\mathcal {H}(A)$ of a finite dimensional algebra $A$ over a finite field. By definition, the Hall algebra $\mathcal {H}(A)$ is a free abelian group with basis the isoclasses (isomorphism classes) of finite dimensional $A$-modules, and the structure constants are given by the so-called Hall numbers, which count the number of certain submodules. Ringel \cite{R90,R90a} proved that if $A$ is representation finite and hereditary, then the twisted Hall algebra $\mathcal {H}_{v}(A)$, called the Ringel--Hall algebra, is isomorphic to the positive part of the corresponding quantized enveloping algebra. Later on, Green \cite{Gr95} introduced a bialgebra structure on $\mathcal {H}_{v}(A)$,
and he showed that the composition subalgebra of $\mathcal {H}_{v}(A)$ generated by simple $A$-modules
provides a realization of the positive part of the corresponding quantized enveloping algebra.

In case that $A$ is representation finite and hereditary, Ringel \cite{R90} showed that the Hall numbers of $A$ are actually integer polynomials in the cardinalities of finite fields. The proof is based on
the directedness of the Auslander--Reiten quiver of the module category of $A$.
These polynomials are called Hall polynomials as in the classical case; see \cite{Macdonald}.
Then one can define the generic Hall algebra $H_{\boldsymbol v}(A)$ over the Laurent polynomial ring $\mathbb{Z}[\boldsymbol v,\boldsymbol v^{-1}]$ and its degeneration $H_1(A)$ at $\boldsymbol v=1$. It was shown by Ringel \cite{R90} that $H_1(A)\otimes\mathbb{C}$ is isomorphic
to the positive part of the universal enveloping algebra of the complex semisimple Lie algebra associated with $A$.
Since then, much subsequent work was devoted to the study of Hall polynomials for various classes
of algebras. In \cite{R91b}, Ringel conjectured that Hall polynomials exist for all representation finite algebras. This conjecture has been proved only for some special algebras, see for example \cite{Fu,GuoP,Nasr,Peng,R90}. By {Ringel \cite{R93}} and Guo \cite{Guo}, Hall polynomials exist in the category of finite dimensional nilpotent representations for a cyclic quiver. Some Hall polynomials for representations of
the Kronecker quiver has also been calculated in \cite{Zhang,Szato}.
More generally, Hubery \cite{Hub} proved that Hall polynomials exist for all tame (affine) quivers with respect
to the decomposition classes of Bongartz and Dudek \cite{BD}.
Recently, Deng and Ruan \cite{DR} have generalized Hubery's result to a more general setting.
They proved that Hall polynomials exist for domestic weighted projective lines with respect to the decomposition sequences.
As an application, they obtained the existence of Hall polynomials for tame quivers.
We remark that the existence of Hall polynomials has gained importance recently by the relevance of quiver Grassmannians with cluster algebras, see \cite{CR}.

In order to realize the entire quantized enveloping algebra via Hall algebra approach, one turns to consider the Hall algebras of triangulated categories (c.f. \cite{Kapranov,Toen,XiaoXu}). T\"oen \cite{Toen} defined a derived Hall algebra $\mathcal {D}\mathcal {H}(\mathcal {A})$ for a differential graded category $\mathcal {A}$ satisfying some finiteness conditions, where the structure constants are given by the so-called derived Hall numbers.
Later on, Xiao and Xu \cite{XiaoXu} investigated the derived Hall algebra for an arbitrary triangulated category under certain finiteness conditions, which includes the bounded derived category $D^b(A)$ of a finite dimensional hereditary algebra $A$ over a finite field.

A natural question is whether the derived Hall numbers also have a generic phenomenon like Hall polynomials for Hall numbers. The aim of this paper is to give a positive answer to this question for Dynkin and tame quivers. It seems that in the Ringel--Hall algebra case, the associativity formula alone can not be applied to solve the problem since the middle term remains unchanged, while it is successful to use Green's Formula to reduce the dimension vector of the middle term and then make induction.
But for the derived Hall algebra case, Green's Formula is not available. This makes the problem much more difficult in this context. Our main strategy is to rotate triangles to change the ``middle terms", and then to use the associativity formula to proceed induction.

The paper is organized as follows. In Section 2 we give a brief introduction to the representation categories of tame quivers, and recall the definitions of derived Hall numbers and derived Hall algebras. In Section 3 we define the generic functions of derived Hall numbers and present our Main Theorem which states that the generic functions exist for Dynkin and tame quivers, we also give some preparatory results which are needed for the proof of the Main Theorem in Sections 4 and 5.

%%%%%%%%%%%%%%%%%%%%%%%%%%%%%%%%%%%%%%%%%%%%%%%%%%%%%%%%%%%%%%%%%%%%%%%%%%%%%%%%%%%%%%%%%%%%%%%%%%%%%%%%%%%%%%%%%%%%%%%%%
\section{Preliminaries}

\subsection{}

Throughout the paper, $k$ denotes a finite field with $q$ elements. Let $Q$ be a finite quiver and let $A=kQ$ be the path algebra of $Q$ over $k$. By $\mod A$ and $D^b(A)$ we denote respectively the category of finite dimensional (left) $A$-modules and its bounded derived category. The suspension functor of $D^b(A)$ is denoted by $[1]$, and we write $\Hom(-,-)$ for $\Hom_{D^b(A)}(-,-)$ to simply the notation.

For a cyclic quiver $Q$ with $m$ vertices, denote by $\mathcal {J}_m$ the category of finite dimensional nilpotent representations of $Q$ over $k$. It is well-known that the set of isoclasses of objects in $\mathcal {J}_m$ is in bijection with $m$-tuples of partitions; see for example \cite{Sch}. In particular, the classification of nilpotent representations is independent of the ground field $k$.

For a tame quiver $Q$, it is well
known from \cite{Dlab} that the subcategory $\ind A$ consisting of indecomposable $A$-modules admits
a disjoint decomposition
$\ind A = \mathcal {P}\cup\mathcal {R}\cup\mathcal {I}$,
where $\mathcal {P}$ (resp., $\mathcal {R}$, $\mathcal {I}$) denotes the subcategory of indecomposable preprojective
(resp., regular, preinjective) $A$-modules. Moreover, they satisfy
$$\;\Hom_A(\I,\R) = 0,\qquad\Hom_A(\I,\P) = 0,\qquad\Hom_A(\R,\P) = 0,$$
$$\;\Ext^1_A(\R,\I) = 0,\qquad\;\;\Ext^1_A(\P,\I) = 0,\qquad\;\Ext^1_A(\P,\R) = 0.$$

{The Auslander--Reiten quiver of $\R$ consists of finitely many non-homogeneous tubes and infinitely many homogeneous tubes. More precisely,
each non-homogeneous tube of rank $m$ is equivalent to the category of nilpotent representations
of a cyclic quiver with $m$ vertices over $k$, while for each homogeneous tube $\mathcal {J}$ with the quasi-simple top $M$,
$k_M:=\End_A(M)$ is a finite field
extension of $k$ and $\mathcal {J}$ is equivalent to the
category of nilpotent representations of the Jordan quiver over $k_M$.}

%%%%%%%%%%%%%%%%%%%%%%%%%%%%%%%%%%%%%%%%%%%%%%%%%%%%%%%%%%%%%%%%%%%%%%%%%%%%%%%%%%%%%%%%%%%%%%%%%%%%%%
\subsection{}

For any three objects $X,Y,L$ in $D^b(A)$, denote by $$W_{XY}^L:=\{(f,g,h)~|~Y\xrightarrow{f}L\xrightarrow{g}X\xrightarrow{h}Y[1]~~\mbox{is~a~triangle~in}~~D^b(A)\}.$$
Consider the action of the group $\Aut Y$ on the set $W_{XY}^L$ defined by
$$\xymatrix{Y\ar[r]^-f\ar[d]^-\eta&L\ar[r]^-g\ar@{=}[d]&X\ar[r]^-h\ar@{=}[d]&Y[1]\ar[d]^-{\eta[1]}\\
Y\ar[r]^-{f'}&L\ar[r]^-{g}&X\ar[r]^-{h'}&Y[1].}$$
Denote the set of orbits by $(W_{XY}^L)_Y^*$. Dually, we consider the action of the group $\Aut X$ on the set $W_{XY}^L$ and obtain the orbit set $(W_{XY}^L)_X^*$.

Since the actions above are not free, in general, $$\frac{|(W_{XY}^L)_Y^*|}{|\Aut X|}\neq\frac{|(W_{XY}^L)_X^*|}{|\Aut Y|}.$$
However, by \cite{XiaoXu} we know that
$$\frac{|(W_{XY}^L)_Y^*|}{|\Aut X|}\cdot\frac{\{L,X\}}{\{X,X\}}=\frac{|(W_{XY}^L)_X^*|}{|\Aut Y|}\cdot\frac{\{Y,L\}}{\{Y,Y\}}=:F_{X,Y}^L,$$
where and elsewhere we denote by $|S|$ the cardinality of a finite set $S$, and for any $M,N\in D^b(A)$, $\{M,N\}:=\prod\limits_{i>0}|\Hom(M[i],N)|^{(-1)^i}$. The number $F_{X,Y}^L$ is called the \emph{derived Hall number} associated to the objects $X,Y,L$ in $D^b(A)$. Actually, it is easy to see that
\begin{equation*}|(W_{XY}^L)_Y^*|=|\Hom(L,X)_{Y[1]}|~~~~\mbox{and}~~~~ |(W_{XY}^L)_X^*|=|\Hom(Y,L)_{X}|,\end{equation*} where $\Hom(M,N)_Z$ denotes the subset of $\Hom(M,N)$ consisting of morphisms $M\to N$ whose cone is isomorphic to $Z$. Namely, we have the following identity
$$F_{X,Y}^L=\frac{|\Hom(L,X)_{Y[1]}|}{|\Aut X|}\cdot\frac{\{L,X\}}{\{X,X\}}=
\frac{|\Hom(Y,L)_{X}|}{|\Aut Y|}\cdot\frac{\{Y,L\}}{\{Y,Y\}}.$$

The \emph{derived Hall algebra} $\mathcal {D}\mathcal {H}(A)$ of $A$ is the $\mathbb{Q}$-space spanned by the isoclasses $[X]$ of objects in $D^b(A)$, and the multiplication is defined by
$$[X]\cdot[Y]=\sum\limits_{[L]}F_{X,Y}^L[L].$$

By \cite[Theorem 3.6]{XiaoXu}, $\mathcal {D}\mathcal {H}(A)$ is an associative and unital algebra.
The associativity of the derived Hall algebra states that for a quadruple $\{X,Y,Z;L\}$ of objects in $\mathcal {D}\mathcal {H}(A)$,
$$\sum\limits_{[W]}F_{X,Y}^WF_{W,Z}^L=\sum\limits_{[U]}F_{X,U}^LF_{Y,Z}^U.$$

Now, we give some simple calculations on the derived Hall numbers.

\begin{lemma}\label{same}
If there exists a hereditary abelian subcategory $\mathcal{B}$ of $D^b(A)$ which is derived equivalent to $\mod A$, such that $X,Y,L\in\mathcal{B}$. Then the derived Hall number $F_{X,Y}^L$ coincides with the Hall number with respect to the triple $\{X,Y,L\}$ of objects in $\mathcal{B}$.
\end{lemma}
\bp
Since $X,L\in\mathcal{B}$, we obtain that $\{L,X\}=\{X,X\}=0$. We claim that the action of $\Aut Y$ on $W_{XY}^L$ is free. Indeed, for any $\eta\in\Aut Y$ and $(f,g,h)\in W_{XY}^L$, assume that $f=f\circ\eta$, that is, $f\circ(1-\eta)=0$, then there is a morphism $x\in\Hom(Y,X[-1])$ such that $1-\eta=-h[-1]\circ x$. Since $\Hom(Y,X[-1])=0$, we conclude that $x=0$ and then $\eta=1$. Hence, $|(W_{XY}^L)_Y^*|=\frac{|W_{XY}^L|}{|\Aut Y|}$. Thus,
$$F_{X,Y}^L=\frac{|W_{XY}^L|}{|\Aut X|\cdot|\Aut Y|}.$$ By \cite[Lemma 2.12]{Xiao2}, we know that there is a bijection between $W_{XY}^L$ and $\{(f,g)~|~0\to Y\xrightarrow{f}L\xrightarrow{g}X\to 0~\mbox{is a short exact sequence in}~\mathcal{B}\}$. Therefore, we complete the proof.
\ep

\begin{lemma}\label{pinfan}
For any objects $X,Y\in D^b(A)$,
$$F_{X,Y}^{X\oplus Y}=\{Y,X\}\cdot\frac{|\Aut\left(X\oplus Y\right)|}{|\Hom\left(X,Y\right)|\cdot|\Aut X|\cdot|\Aut Y|}.$$ In particular, if $X$ and $Y$ have no common direct summands, then
$$F_{X,Y}^{X\oplus Y}=\{Y,X\}\cdot|\Hom(Y,X)|.$$
\end{lemma}
\bp
By definition, $$F_{X,Y}^{X\oplus Y}=\frac{|(W_{XY}^{X\oplus Y})_Y^*|}{|\Aut X|}\cdot\frac{\{X\oplus Y,X\}}{\{X,X\}}=\{Y,X\}\cdot\frac{|(W_{XY}^{X\oplus Y})_Y^*|}{|\Aut X|}.$$
We claim that the action of $\Aut Y$ on $W_{XY}^{X\oplus Y}$ is free. Indeed, for any $\eta\in\Aut Y$ and $\left({{f_1}\choose{f_2}},(g_1,g_2),0\right)\in W_{XY}^{X\oplus Y}$, if ${{f_1}\choose{f_2}}={{f_1}\choose{f_2}}\eta$, that is, ${{f_1}\choose{f_2}}(1-\eta)=0$, then $1-\eta$ factors through the zero morphism $X[-1]\xrightarrow{0} Y$, so $\eta=1$. By \cite[Lemma 8]{Hubery}, we know that $|W_{XY}^{X\oplus Y}|=\frac{|\Aut\left(X\oplus Y\right)|}{|\Hom\left(X,Y\right)|}$. Hence,
$$|(W_{XY}^{X\oplus Y})_Y^*|=\frac{|\Aut\left(X\oplus Y\right)|}{|\Hom\left(X,Y\right)|\cdot|\Aut Y|}.$$
Thus, $$F_{X,Y}^{X\oplus Y}=\{Y,X\}\cdot\frac{|\Aut\left(X\oplus Y\right)|}{|\Hom\left(X,Y\right)|\cdot|\Aut X|\cdot|\Aut Y|}.$$
In particular, if $X$ and $Y$  have no direct summands in common, then $|\Aut(X\oplus Y)|=|\Aut X|\cdot|\Aut Y|\cdot|\Hom(X,Y)|\cdot|\Hom(Y,X)|$. Therefore, we complete the proof.
\ep

\begin{lemma}\label{xuanzhuan}
Let $X,Y,L\in D^b(A)$. Then

$(1)$~$F_{L,X[-1]}^Y=\frac{|\Aut Y|\cdot|\{Y,Y\}|}{|\Aut L|\cdot|\{L,L\}|}\cdot F_{X,Y}^L$;~~~~\quad\quad
$(2)$~$F_{Y[1],L}^{X}=\frac{|\Aut X|\cdot|\{X,X\}|}{|\Aut L|\cdot|\{L,L\}|}\cdot F_{X,Y}^L$.
\end{lemma}
\bp
$(1)$~By definition, \begin{equation*}\begin{split}F_{X,Y}^L&=\frac{|(W_{XY}^L)_X^*|}{|\Aut Y|}\cdot\frac{\{Y,L\}}{\{Y,Y\}}\\&=\frac{|(W_{L,X[-1]}^Y)_{X[-1]}^*|}{|\Aut L|}\cdot\frac{\{Y,L\}}{\{L,L\}}\cdot\frac{|\Aut L|\cdot\{L,L\}}{|\Aut Y|\cdot\{Y,Y\}}\\&=
F_{L,X[-1]}^Y\cdot\frac{|\Aut L|\cdot\{L,L\}}{|\Aut Y|\cdot\{Y,Y\}}.
\end{split}\end{equation*}

$(2)$~Similar to (1).
\ep

The following lemma is frequently-used in a triangulated category.

\begin{lemma}\rm{(\cite[Lemma 2.5]{Xiao2})}\label{youyong}
Let $\xymatrix{N \ar[r]^-{{f_1}\choose{f_2}}& L_1\oplus L_2\ar[r]^-{(g_1,g_2)}&M\ar[r]^-{h}&N[1]}$ be a triangle in a triangulated category. If $f_2=0$, then this triangle is isomorphic to the triangle
$$\xymatrix{N \ar[r]^-{{f_1}\choose{0}}& L_1\oplus L_2\ar[r]^-{\left({\begin{smallmatrix}g'&0\\0&1\end{smallmatrix}}\right)}&M'\oplus L_2\ar[r]^-{(h',0)}&N[1],}$$ where $\xymatrix{N\ar[r]^-{f_1}&L_1\ar[r]^-{g'}&M'\ar[r]^-{h'}&N[1]}$ is also a triangle. A similar result also holds whenever $g_2=0$.
\end{lemma}

\begin{lemma}\label{xiaoqu}
Let $M, N, L_1, L_2\in D^b(A)$.

$(1)$ If $\Hom(N,L_2)=0$, then
$F_{M,N}^{L_1\oplus L_2}=F_{M',N}^{L_1}\cdot{\{N,L_2\}},$ whenever $M'\oplus L_2\cong M$.

$(2)$ If $\Hom(L_2,M)=0$, then
$F_{M,N}^{L_1\oplus L_2}=F_{M,N'}^{L_1}\cdot\{L_2,M\}$, whenever $N'\oplus L_2\cong N$.

\end{lemma}

\begin{proof}
$(1)$~Since $\Hom(N,L_2)=0$, we get that every orbit in $(W_{MN}^{L_1\oplus L_2})_M^*$ has a representative
$\left({{f_1}\choose{0}},\left({\begin{smallmatrix}g'&\\&1\end{smallmatrix}}\right),(h',0)\right)$ as that in Lemma \ref{youyong}. It is easy to see that there is a bijection from $(W_{MN}^{L_1\oplus L_2})_M^*$ to $(W_{M'N}^{L_1})_{M'}^*$. Hence,
\begin{equation*}\begin{split}F_{M,N}^{L_1\oplus L_2}&=\frac{|(W_{MN}^{L_1\oplus L_2})_M^*|}{|\Aut N|}\cdot\frac{\{N,L_1\oplus L_2\}}{\{N,N\}}\\
&=\frac{|(W_{M'N}^{L_1})_{M'}^*|}{|\Aut N|}\cdot\frac{\{N,L_1\}}{\{N,N\}}\cdot{\{N,L_2\}}\\
&=F_{M',N}^{L_1}\cdot{\{N,L_2\}}.\end{split}\end{equation*}
$(2)$~Similar to (1).
\end{proof}

%%%%%%%%%%%%%%%%%%%%%%%%%%%%%%%%%%%%%%%%%%%%%%%%%%%%%%%%%%
\subsection{}

For any field extension $E$ of $k$, let $A^E=A\otimes_kE$. For any stalk complex $X$ in $D^b(A)$, $X^E=X\otimes_kE$ is well-defined and it is still a stalk complex in $D^b(A^E)$.
Let $X$ be an arbitrary object in $D^b(A)$, writing $X=X_1\oplus\cdots\oplus X_t$ with all $X_i$ indecomposable, we define $X^E=X_1^E\oplus\cdots\oplus X_t^E$. We denote by $D^b(A)^E$ the full subcategory of $D^b(A^E)$ whose objects are all $X^E$ with $X\in D^b(A)$.

A finite field extension $E$ of $k$ is said to
be \emph{conservative} relative to $X\in D^b(A)$ if for each indecomposable direct summand $Y$ of $X$, $Y^E$ is
indecomposable in $D^b(A^E)$. In general, given a finite set $S = \{X_1, . . . ,X_t\}$ of
objects in $D^b(A)$, a finite field extension $E$ of $k$ is said to be \emph{conservative} relative to
$S$ if $E$ is conservative relative to each $X_i$ for $1\leq i\leq t$. Note that there
exist infinitely many conservative field extensions of $k$ relative to $S$.

By \cite[Lemma 7.4]{Lam}, for any $X,Y\in\mod A$ and any conservative field extension $E$ of $k$ relative to $\{X,Y\}$, we have
\begin{equation}\label{Hom1}\Hom_{A^E}(X^E,Y^E)\cong\left(\Hom_A(X,Y)\right)^E.\end{equation}

%%%%%%%%%%%%%%%%%%%%%%%%%%%%%%%%%%%%%%%%%%%%%%%%%%%%%%%%%%%%%%%%%%%%%%%%%%%%%%%%%%%%%%%%%%%%%%%%%%%%%%%%%%%%%%%%%%%%%%%%%%%%%%%%%
\section{Main result}

\begin{definition} {Let $X, Y, L\in D^b(A)$, if there exists a rational function $\varphi_{X,Y}^L\in\mathbb{Q}(T)$ such that for each conservative field extension $E$ of
$k$ relative to $\{X,Y,L\}$,
$$\varphi_{X,Y}^L(|E|)=F_{X^E,Y^E}^{L^E},$$
then we say that the generic function $\varphi_{X,Y}^L$ exists for $\{X, Y, L\}$.
If the generic function $\varphi_{X,Y}^L$ exists for all $X, Y, L\in D^b(A)$, then we say that the generic functions exist for $A$.}
\end{definition}

\begin{remark} By definition, for any exceptional indecomposable $A$-module $X$,
$$\varphi_{X[1],X}^0=\varphi_{X,X[-1]}^0=\frac{1}{T-1}.$$
Hence, quite different from the Hall number case, we can only expect the
derived Hall numbers behave as a rational function rather than a polynomial.
\end{remark}

\noindent$\mathbf{Main~~Theorem}$~~~~
\emph{Let $Q$ be a Dynkin or tame quiver. The generic functions exist for $A=kQ$.}

The proof of this theorem will be given throughout the next two sections, which are devoted to proving the case that $Q$ is of cyclic type and the case that $Q$ is of Dynkin or tame type, respectively. First of all, let us give some preparations.

\begin{lemma}\label{homaut}
For any $X,Y\in D^b(A)$, $\dim_k\Hom(X,Y)$ is independent of the ground field $k$, and there exists a monic polynomial $\mathfrak{a}(T)\in\mathbb{Z}[T]$ such that for each conservative field extension $E$ of $k$ relative to $X$, we have $$\mathfrak{a}(|E|)=|\Aut X^E|.$$
\end{lemma}

\bp
Write $X=\bigoplus\limits_{i\in\mathbb{Z}}X_i[i]$ and $Y=\bigoplus\limits_{i\in\mathbb{Z}}Y_i[i]$ in $D^b(A)$, where all $X_i$ and $Y_i$ are $A$-modules. By \cite[(2.1.3)]{Kapranov}, we have
\begin{equation}\label{Hom2}
\begin{split}\Hom(X,Y)&\cong\bigoplus\limits_{i\in\mathbb{Z}}\Hom_A(X_{-i},Y_{-i})\oplus
\bigoplus\limits_{i\in\mathbb{Z}}\Ext_A^1(X_{-i},Y_{-i+1})\\
&\cong\bigoplus\limits_{i\in\mathbb{Z}}\Hom_A(X_{-i},Y_{-i})\oplus
\bigoplus\limits_{i\in\mathbb{Z}}\D\Hom_A(Y_{-i+1},\tau X_{-i}),\end{split}\end{equation} where $\tau$ is the Auslander-Reiten translation for $\mod A$, and $\D$ is the linear duality. By (\ref{Hom1}), the first statement is proved.
Using arguments similar to those in \cite[Section 2]{R90}, we prove the second statement.
\ep

\begin{lemma}\label{dkh}
For any $X,Y\in D^b(A)$ and any conservative field extension $E$ of $k$ relative to $\{X,Y\}$, $\{X^E,Y^E\}$ is a power of $|E|$.
\end{lemma}
\bp
By definition, $$\{X^E,Y^E\}=\prod\limits_{i>0}|\Hom_{D^b(A^E)}(X^E[i],Y^E)|^{(-1)^i}=|E|^{\sum_{i>0}(-1)^i\rm{dim}_k\rm{Hom}(X[i],Y)}.$$
By Lemma \ref{homaut}, we complete the proof.
\ep

In what follows we show that the Main Theorem holds for some special triples $\{X,Y,L\}$ of objects in $D^b(A)$.

\begin{lemma}\label{generic function for direct sum}
For any $X,Y\in D^b(A)$, the generic function $\varphi_{X,Y}^{X\oplus Y}$ exists.
\end{lemma}

\bp
This follows from Lemma \ref{pinfan}, Lemma \ref{homaut} and Lemma \ref{dkh}.
\ep

Combining Lemma $\ref{same}$ with \cite[Corollary 6.4]{DR}, we have the following result.

\begin{lemma}\label{Hall}
Let $X, Y, L\in \mod A$. Then there exists a polynomial $\varphi_{X,Y}^L\in\mathbb{Z}[T]$ such that for each conservative field extension $E$ of
$k$ relative to $\{X,Y,L\}$,
$$\varphi_{X,Y}^L(|E|)=F_{X^EY^E}^{L^E}.$$
\end{lemma}

\begin{remark}\label{reduce to compact case}
For a given triple $\{X,Y,L\}$ of objects in $D^b(A)$, in order to consider the existence of the generic function $\varphi_{X,Y}^L$, we will always assume that $F_{X,Y}^{L}\neq 0$ from now onwards, otherwise we can take $\varphi_{X,Y}^L=0$. Moreover, taking into account Lemmas \ref{xiaoqu} and \ref{dkh}, we further assume that $\Hom(L_i, X)\neq 0$ and $\Hom(Y, L_i)\neq 0$ for any non-zero direct summand $L_i$ of $L$.
\end{remark}

%%%%%%%%%%%%%%%%%%%%%%%%%%%%%%%%%%%%%%%%%%%%%%%%%%%%%%%%%%%%%%%%%%%%%%%%%%%%%%%%%%%%%%%%%%%%%%%%%%%%%%%%%%%%%%%%%%%%%%%%%%%%%%%%%%%%%%%%%%%%%
\section{The cyclic quiver case}

In this section, we will prove the Main Theorem for the case that $Q$ is a cyclic quiver. Let $\mathcal {J}$ and $D^b(\mathcal{J})$ be the category of finite dimensional nilpotent $k$-representations of the cyclic quiver $Q$ and its bounded derived category, respectively. It is known that Hall polynomials exist for $\mathcal {J}$; see for example \cite{Guo,Hub,R93}.

We first consider the following special case:

\begin{lemma}\label{cyclic and Y ind}
Let $Y$ be an indecomposable object in $D^b(\mathcal{J})$. Then the generic function $\varphi_{X,Y}^L$ exists for any $X,L\in D^b(\mathcal {J})$.
\end{lemma}

\bp
Without loss of generality, we assume that $Y\in\mathcal {J}$. Then by Remark \ref{reduce to compact case} we can assume that $X=X_0\oplus X_1[1]$ and $L=L_0\oplus L_1[1]$ with $X_i, L_i\in\mathcal{J}$ for $i=0,1$. Then the associativity of derived Hall algebras for $\{X_0,X_1[1],Y;L\}$ produces that
\begin{equation}\label{F0}F_{X_0,X_1[1]}^{X}F_{X,Y}^{L}=\sum\limits_{[U]}F_{X_0,U}^{L}F_{X_1[1],Y}^{U}.\end{equation}
Note that $F_{X_0,X_1[1]}^{X}=1$ by Lemma \ref{pinfan}, and from $F_{X_0,U}^{L}\neq 0$ and $F_{X_1[1],Y}^U\neq0$ we can write $U=U_0\oplus L_1[1]$ for some $U_0\in\mathcal {J}$. Hence, by Lemma \ref{xiaoqu} we have
\begin{equation}\label{F1}F_{X_0,U_0\oplus L_1[1]}^{L_0\oplus L_1[1]}=F_{X_0,U_0}^{L_0}\cdot\{L_1[1],X_0\}=F_{X_0,U_0}^{L_0};\end{equation}
and by the proof of \cite[Proposition 7.1]{Toen} and \cite[(8.8)]{Van} we obtain that
\begin{equation}\label{F2}F_{X_1[1],Y}^{U_0\oplus L_1[1]}=q^{-\lr{U_0, L_1}}\cdot\frac{|\Aut L_1|\cdot|\Aut U_0|}{|\Aut X_1|\cdot|\Aut Y|}\sum\limits_{[L']:L'\in\mathcal {J}}|\Aut L'|\cdot F_{L',L_1}^{X_1}F_{U_0,L'}^Y,\end{equation}where
$\lr{U_0, L_1}:=\dim_k \Hom_{\mathcal{J}}(U_0,L_1)-\dim_k \Ext^1_{\mathcal{J}}(U_0,L_1)$ denotes the Euler form of $\mathcal{J}$, which is independent of the ground field.
By the fact that Hall polynomials exist for $\mathcal{J}$ and Lemma \ref{homaut}, the proof is finished. \ep

In order to prove the existence of the generic function $\varphi_{X,Y}^L$ for any $X,Y,L\in D^b(\mathcal {J})$, we need the following general result in triangulated categories.

\begin{lemma}\label{Ext}
Let $$\xymatrix{\xi: X\ar[r]^-u&Y\ar[r]^-v&Z\ar[r]^-w&TX}$$ be a triangle in a $k$-linear triangulated category $(\mathcal {C},T)$. Then for any object $M\in\mathcal{C}$,
\begin{itemize}
\item[$(1)$] $\dim_k\Hom(M,TY)\leq\dim_k\Hom(M,TX\oplus TZ)$ and the strict inequality holds if there exists $f\in\Hom(M,Z)$ such that $wf\neq0$.
\item[$(2)$] $\dim_k\Hom(Y,TM)\leq\dim_k\Hom(X\oplus Z,TM)$ and the strict inequality holds if there exists $g\in\Hom(X,M)$ such that $(Tg)w\neq0$.
\end{itemize}
In particular, \begin{equation}\label{denghao}\dim_k\Hom(Y,TY)\leq\dim_k\Hom(X\oplus Z,TX\oplus TZ)\end{equation}
and the equality holds if and only if the triangle $\xi$ is split, equivalently, $Y\cong X\oplus Z$.
\end{lemma}
\bp
This proof is an adaptation of the proof in \cite[Lemma 2.1]{GuoP}.

$(1)$ Applying $\Hom(M,-)$ to the triangle $\xi$, we obtain a long exact sequence
$$\cdots\to\Hom(M,Z)\xrightarrow{w_\ast}\Hom(M,TX)\to\Hom(M,TY)\to\Hom(M,TZ)\to\cdots.$$
It follows that \begin{equation*}\begin{split}
\dim_k\Hom(M,TY)&\leq\dim_k\Hom(M,TX)+\dim_k\Hom(M,TZ)-\dim_k\im w_\ast\\
&=\dim_k\Hom(M,TX\oplus TZ)-\dim_k\im w_\ast\\
&\leq\dim_k\Hom(M,TX\oplus TZ).
\end{split}
\end{equation*}
If there exists $f\in\Hom(M,Z)$ such that $wf\neq0$, that is, $w_\ast$ is not zero and thus the last inequality above must be strict.

$(2)$ This is proved in a similar way to $(1)$.

In particular, taking $M=Y$ and $M=X\oplus Z$ in $(1)$ and $(2)$, respectively, we obtain
\begin{equation}\label{dengshi}\dim_k\Hom(Y,TY)\leq\dim_k\Hom(Y,TX\oplus TZ)\leq\dim_k\Hom(X\oplus Z,TX\oplus TZ).\end{equation}
If the equality in $(\ref{denghao})$ holds, then the two equalities in $(\ref{dengshi})$ both hold. So by $(2)$ for $g={1\choose 0}: X\to X\oplus Z$, we have $(Tg)w=0$. Hence, $w=0$ and thus $\xi$ is split. Conversely, if $\xi$ is split, then clearly the equality in $(\ref{denghao})$ holds. The last equivalence follows from \cite[Lemma 3]{PX2000}.
\ep

For each object $M\in D^b(\mathcal{J})$, we define $d(M)$ to be the number of indecomposable direct summands of $M$,  and set $l(M)=\dim_k\Hom(M,M[1])$.

\begin{lemma}\label{lx0}
For any $Y\in D^b(\mathcal{J})$ with $l(Y)=0$, the generic function $\varphi_{X,Y}^L$ exists for any $X,L\in D^b(A)$.
\end{lemma}

\bp
We proceed the proof by induction on $d(Y)$. If $d(Y)\leq 1$, then by Lemma \ref{cyclic and Y ind} $\varphi_{X,Y}^L$ exists.
So we assume that $d(Y)>1$, and write $Y=Y_1\oplus Y_2$, where $Y_i\in D^b(A)$ and $d(Y_i)<d(Y)$ for $i=1,2$. Moreover, $l(Y)=0$ implies $\Hom(Y_1,Y_2[1])=0$. Hence, by the associativity for $\{X,Y_1,Y_2;L\}$, we have
\begin{equation}\label{jiehex1}F_{X,Y}^LF_{Y_1,Y_2}^{Y}=\sum\limits_{[W]}F_{X,Y_1}^WF_{W,Y_2}^L.\end{equation}
Note that the generic function $\varphi_{Y_1,Y_2}^{Y}$ already exists by Lemma \ref{generic function for direct sum}, and the sum on the right-hand side is taken over a finite set which is independent of field extensions.
Furthermore, by induction we know that the generic functions $\varphi_{X,Y_1}^W, \varphi_{W,Y_2}^L$ exist for any $W$. This completes the proof.
\ep

Let us give an order on the set $\{(l,d)~|~l\in\mathbb{N},d\in\mathbb{N}^+\}$ defined by
$$(l,d)\leq(l',d')~\Longleftrightarrow~l<l'~\mbox{or}~l=l'~\mbox{and}~ d\leq d'.$$

\textbf{\emph{Proof of the Main Theorem for $Q$ of cyclic type:}}

We will show that the generic function $\varphi_{X,Y}^L$ exists for any $X,Y,L\in D^b(\mathcal {J})$ by induction on $(l(Y),d(Y))$. If $l(Y)=0$ or $d(Y)=1$, then by Lemma \ref{lx0} or Lemma \ref{cyclic and Y ind} respectively, we are done. So we assume that $l(Y)>0$ and $d(Y)>1$, and write $Y=Y_1\oplus Y_2$, where $Y_1,Y_2\in D^b(A)$ are nonzero. Then for $i=1,2$, $l(Y_i)\leq l(Y)$ and $d(Y_i)<d(Y)$, thus
$(l(Y_i),d(Y_i))<(l(Y),d(Y)).$

By the associativity for $\{X,Y_1,Y_2;L\}$, we have
$$\sum\limits_{[W]}F_{X,Y_1}^WF_{W,Y_2}^L=\sum\limits_{[U]}F_{X,U}^LF_{Y_1,Y_2}^U.$$
Hence, \begin{equation}\label{zh}F_{X,Y}^LF_{Y_1,Y_2}^{Y}
=\sum\limits_{[W]}F_{X,Y_1}^WF_{W,Y_2}^L-\sum\limits_{[U]:U\not\cong Y}F_{X,U}^LF_{Y_1,Y_2}^U.\end{equation}
Note that $\varphi_{Y_1,Y_2}^{Y}$ already exists by Lemma \ref{generic function for direct sum}, and the sums on the right-hand side are both taken over finite sets which are independent of field extensions.
Moreover, by Lemma \ref{Ext}, $l(U)<l(Y)$ for each $U$ in the second sum. Thus, by induction, the generic functions exist for all the derived Hall numbers appearing on the right-hand side of (\ref{zh}). We then finish the proof.

%%%%%%%%%%%%%%%%%%%%%%%%%%%%%%%%%%%%%%%%%%%%%%%%%%%%%%%%%%%%%%%%%%%%%%%%%%%%%%%%%%%%%%%%%%%%%%%%%%%%%%%%%%%%%%%%%%%%%%%%%%%%%%%%%
\section{The general tame quiver cases}

In this section, we are going to prove the Main Theorem for $Q$ of Dynkin or tame type. {We will show that the generic function $\varphi_{X,Y}^L$ exists by discussing with $X$.} We first consider $X$ as an indecomposable $A$-module, then a decomposable $A$-module and finally an arbitrary object in $D^b(A)$.

\begin{lemma}\label{case of X ind non regular mod}
Let $X$ be an indecomposable preprojective or preinjective $A$-module. Then the generic function $\varphi_{X,Y}^L$ exists for any $Y,L\in D^b(A)$.
\end{lemma}

\bp
Without loss of generality, we can assume that $X$ is a simple injective $A$-module. (Indeed, we can choose a tilting object $T$ in $D^b(A)$, such that the endomorphism algebra $\End(T)$ is derived equivalent to $A$ and $X$ viewed as an $\End(T)$-module is simple and injective.)
Then according to Remark \ref{reduce to compact case}, we can assume that both of $Y, L$ belong to the category $\mod A$. Now the result follows from Lemma \ref{Hall}.
\ep

\textbf{\emph{Proof of the Main Theorem for $Q$ of Dynkin type:}}

We prove by induction on $d(X)$, which denotes the number of indecomposable direct summands of $X$. If $d(X)\leq 1$, then by Lemma
\ref{case of X ind non regular mod} we are done.
If $d(X)>1$, then by the directedness of the Auslander-Reiten quiver of $D^b(kQ)$, we can decompose $X$ as $X=X_1\oplus X_2$, where $d(X_i)<d(X)$ for $i=1,2$ and $\Hom(X_1, X_2[1])=0$. Hence, by the associativity for $\{X_1,X_2,Y;L\}$, we have
\begin{equation}\label{jiehex2}F_{X_1,X_2}^{X}F_{X,Y}^L=\sum\limits_{[U]}F_{X_1,U}^LF_{X_2,Y}^U.\end{equation}
Note that $\varphi_{X_1,X_2}^{X}$ already exists by Lemma \ref{generic function for direct sum}, and the sum on the right-hand side is taken over a finite set which is independent of field extensions.
Furthermore, by induction, for each $U$ the generic functions $\varphi_{X_1,U}^L, \varphi_{X_2,Y}^U$ exist. This completes the proof.

Now we turn to prove the Main Theorem for the general tame cases. In what follows, let $Q$ be an acyclic tame quiver.

\begin{proposition}\label{case of X ind mod}
The generic function $\varphi_{X,Y}^L$ exists for $X,Y,L\in D^b(A)$ with $X$ indecomposable.
\end{proposition}

\bp
Without loss of generality, we assume that $X\in\mod A$. If $X$ is preprojective or preinjective, then we are done by
Lemma \ref{case of X ind non regular mod}. Now we assume that $X$ is an indecomposable regular $A$-module which lies in a tube $\mathcal {J}$.
According to Remark \ref{reduce to compact case} we assume that $L=L_{-1}[-1]\oplus L_0$ and $Y=Y_{-1}[-1]\oplus Y_0$ with $Y_i, L_i\in\mod A$ for $i=-1,0$. Since different tubes are $\Hom$-orthogonal and $\Ext$-orthogonal, we can further assume that $L=M_{-1}[-1]\oplus M_f\oplus M_0$ and $Y=N_{-1}[-1]\oplus N_f\oplus N_0$, where $M_i, N_i\in\mathcal{J}, i=-1,0$, and $M_f, N_f\in\I[-1]\cup\P$. Since $\Hom(N_{-1}[-1]\oplus N_f,N_0[1])=0$, by Lemma \ref{pinfan} $F_{N_{-1}[-1]\oplus N_f,N_0}^Y=1$. According to the associativity for $\{X,N_{-1}[-1]\oplus N_f,N_0;L\}$, we have
$$F_{X,Y}^L=\sum\limits_{[W]}F_{X,N_{-1}[-1]\oplus N_f}^WF_{W,N_0}^L.$$

For any $W$ on the right-hand side, $F_{X,N_{-1}[-1]\oplus N_f}^W\neq 0$ implies that $W$ has the form $W=W_{-1}[-1]\oplus W_f\oplus W_0$, where $W_{-1}, W_0\in\mathcal{J}$ and $W_f\in\I[-1]\cup\P$; moreover, for any triangle
$$\xymatrix{N_0\ar[r]&M_{-1}[-1]\oplus M_f\oplus M_0\ar[r]&W_{-1}[-1]\oplus W_f\oplus W_0\ar[r]&N_0[1],}$$
we deduce from $\Hom(W_{-1}[-1]\oplus W_f,N_0[1])=0=\Hom(N_0,M_{-1}[-1]\oplus M_f)$
that $W_{-1}=M_{-1}$, $W_f=M_f$, and notice that $\{N_0,M_{-1}[-1]\oplus M_f\}=1$. It follows that
\begin{equation}\label{tjd}F_{X,Y}^L=\sum\limits_{[W_0]:W_0\in\mathcal{J}}F_{X,N_{-1}[-1]\oplus N_f}^{M_{-1}[-1]\oplus M_f\oplus W_0}F_{W_0,N_0}^{M_0}.\end{equation}

Now the sum on the right-hand side is taken over a finite set which is independent of field extensions, and
$F_{W_0,N_0}^{M_0}$ is given by a polynomial. It suffices to show that for each $W_0$ in (\ref{tjd}) the generic function $\varphi_{X,N_{-1}[-1]\oplus N_f}^{M_{-1}[-1]\oplus M_f\oplus W_0}$ exists.
By the associativity for $\{X,N_{-1}[-1],N_f;M_{-1}[-1]\oplus M_f\oplus W_0\}$ , we have
\begin{equation}\label{F3}F_{X,N_{-1}[-1]\oplus N_f}^{M_{-1}[-1]\oplus M_f\oplus W_0}=\sum\limits_{[U]}F_{X,N_{-1}[-1]}^UF_{U,N_f}^{M_{-1}[-1]\oplus M_f\oplus W_0},\end{equation}
where we have used $F_{N_{-1}[-1],N_f}^{N_{-1}[-1]\oplus N_f}=1$.

Note that each $U$ in (\ref{F3}) is determined by the kernel and cokernel of an morphism $f:X\to N_{-1}$, this implies that $U$ belongs to $\mathcal{J}[-1]\cup \mathcal{J}$. So we write $U=U_{-1}[-1]\oplus U_0$ for some $U_{-1}, U_0\in\mathcal{J}$. Moreover, by considering the triangle
$$\xymatrix{N_f\ar[r]&M_{-1}[-1]\oplus M_f\oplus W_0\ar[r]&U_{-1}[-1]\oplus U_0\ar[r]&N_f[1]}$$
and observing that $\Hom(U_{-1}[-1],N_{f}[1])=0=\Hom(N_f,M_{-1}[-1])$, we conclude that $U_{-1}=M_{-1}$.
Since $\{N_f,M_{-1}[-1]\}=1$, we obtain that $F_{U,N_f}^{M_{-1}[-1]\oplus M_f\oplus W_0}=F_{U_0,N_f}^{M_f\oplus W_0}$ and thus \begin{equation}\label{F4}F_{X,N_{-1}[-1]\oplus N_f}^{M_{-1}[-1]\oplus M_f\oplus W_0}=\sum\limits_{[U_0]:U_0\in\mathcal{J}}F_{X,N_{-1}[-1]}^{U_0\oplus M_{-1}[-1]} F_{U_0,N_f}^{M_f\oplus W_0}.
\end{equation}

The sum on the right-hand side is again taken over a finite set which is independent of field extensions. Since all of
$X, N_{-1} , U_0, M_{-1}$ belong to the category $D^b(\mathcal{J})$, we have proved that $\varphi_{X,N_{-1}[-1]}^{U_0\oplus M_{-1}[-1]}$ exists  in Section 4. Moreover, all of $U_0,W_0,M_f,N_f$ belong to the subcategory $\I[-1]\cup \P \cup\R$ of $D^b(A)$, which is equivalent to the category of coherent sheaves over a domestic weighted projective line. By \cite[Corollary 4.1]{DR}, $\varphi_{U_0,N_f}^{M_f\oplus W_0}$ also exists. This completes the proof.
\ep

Now we turn to consider the generic function $\varphi_{X,Y}^L$ for $X,Y,L\in D^b(A)$ with $X$ decomposable. By degree shift we can always assume that $X=\underset{0\leq i\leq n}{\mathop\bigoplus} X_i[i]$ for some $n\geq 0$ with all $X_i\in\mod A$. According to Remark \ref{reduce to compact case} we then assume  $Y=\underset{-1\leq i\leq n}{\mathop\bigoplus }Y_i[i]$ and $L=\underset{-1\leq i\leq n}{\mathop\bigoplus}L_i[i]$ with all $Y_i, L_i\in\mod A$. The following Lemma plays a key role in the subsequent proofs.

\begin{lemma}\label{diyi}
For any $X_1, X_2, Y, L\in D^b(A)$, write $X_j=\underset{0\leq i\leq n}{\mathop\bigoplus} X_{j,i}[i]$ with all $X_{j,i}\in\mod A$ for $j=1,2$. Then the set $$\mathcal{S}:=\{~[U]~|~U\in D^b(A),F_{X_2,Y}^U\neq0, F_{X_1,U}^L\neq0~\}$$ is independent of the conservative field extension $E$ of $k$ relative to $\{X_1, X_2, Y, L\}$, provided one of the following conditions holds:
\begin{itemize}
\item[(i)] $X_1=X_{1,0}\in\mathcal{P}\cup\mathcal{R}$ and $X_{2,0}\in\mathcal{R}\cup\mathcal{I}$;
\item[(ii)] $X_{1,n}\in\mathcal{P}\cup\mathcal{R}$ and $X_2=X_{2,n}[n]\in\mathcal{R}[n]\cup\mathcal{I}[n]$.
\end{itemize}
\end{lemma}

\bp
Assume Condition ${\rm(i)}$ holds. For any $[U]\in \mathcal{S}$, we show that the isoclasses $[V]$ of each indecomposable direct summand $V$ of $U$ is independent of the conservative field extensions. For this we first consider the triangle $X_1[-1]\to U\to L\to X_1$. If $\Hom(X_1[-1],V)=0$, then $V$ is a direct summand of $L$, we are done. So we assume that $\Hom(X_1[-1],V)\neq 0$. Then by the assumption $X_1\in\mod A$, we can write $U=U_{-1}[-1]\oplus U_0$ with $U_{-1}, U_0\in\mod A$.\\
\textbf{\emph{Case 1:}} $V$ is a direct summand of $U_0$. Then $\Hom(X_1[-1],V)\cong\D\Hom(\tau^{-1}V,X_1)\neq0$, which implies $V\in\mathcal{P}$, or $V$ is regular and lies in a tube which contains a non-zero summand of $X_{1}$, thus $[V]$ is independent of the conservative field extension $E$ of $k$ relative to $\{X_{1}\}$.\\
\textbf{\emph{Case 2:}} $V$ is a direct summand of $U_{-1}[-1]$. We consider the triangle $Y\to U\to X_2\to Y[1]$.
If $\Hom(V,X_2)=0$, then $V$ is a direct summand of $Y$; if $\Hom(V,X_2)\neq0$, thus $\Hom(X_{2,0},\tau V[1])\neq0$, then $V[1]\in\mathcal{I}$, or $V[1]$ is regular and lies in a tube which contains a non-zero summand of $X_{2,0}$, thus $[V]$ is independent of the conservative field extension $E$ of $k$ relative to $\{X_{2,0}, Y\}$.

Therefore, we finish the proof of the result under Condition ${\rm(i)}$. Dually, we can show the result under Condition ${\rm(ii)}$.
\ep

\begin{lemma}\label{case of X non regular mod}
Let $X$ be an preprojective or preinjective $A$-module. Then the generic function $\varphi_{X,Y}^L$ exists for any $Y,L\in D^b(A)$.
\end{lemma}

\bp
We only need to prove the result for $X$ preinjective, since for $X$ preprojective, we can choose a tilting object $T$ in $D^b(A)$, such that the endomorphism algebra $\End(T)$ is derived equivalent to $A$ and $X$ viewed as an $\End(T)$-module is preinjective.
Now, we assume $X\in\mathcal {I}$. Then by Remark \ref{reduce to compact case}, we also assume that $Y,L\in\mathcal{I}[-1]\cup\mod A$. Writing $Y=Y_1\oplus Y_2$ with $Y_1\in\mathcal{I}[-1]\cup\mathcal{P}$ and $Y_2\in\mathcal{R}\cup\mathcal{I}$, by the associativity for $\{X,Y_1,Y_2;L\}$, we obtain that
\begin{equation}F_{X,Y}^LF_{Y_1,Y_2}^Y=\sum\limits_{[W]}F_{X,Y_1}^WF_{W,Y_2}^L.\end{equation}
From $F_{X,Y_1}^W\neq 0$ we get $W\in\mathcal{I}[-1]\cup\mod A$. Write $W=W_{-1}[-1]\oplus W_0$ and $L=L_{-1}[-1]\oplus L_0$ with $W_{-1},L_{-1}\in\mathcal{I}$ and $W_0,L_0\in\mod A$.
Consider the triangle $Y_2\to L\to W\to Y_2[1]$, since $\Hom(Y_2,L_{-1}[-1])=0=\Hom(W_{-1}[-1],Y_2[1])$, we obtain $W_{-1}=L_{-1}$, and then
\begin{equation}\label{jh}F_{X,Y}^LF_{Y_1,Y_2}^Y=\sum\limits_{[W_0]}F_{X,Y_1}^{L_{-1}[-1]\oplus W_0}F_{W_0,Y_2}^{L_0}.\end{equation}
Moreover, for each indecomposable direct summand $V_0$ of $W_0$, if $\Hom(V_0,Y_2[1])=0$, then $V_0$ is a direct summand of $L$; if $\Hom(V_0,Y_2[1])=\Hom(Y_2,\tau V_0)\neq0$, then $V_0\in\mathcal{I}$ or $V_0$ is regular and belongs to a tube which contains a non-zero summand of $Y_{2}$. Therefore, the sum in (\ref{jh}) is taken over a finite set which is independent of the relative conservative field extensions.
Since the generic functions $\varphi_{Y_1,Y_2}^Y$ and $\varphi_{W_0,Y_2}^{L_0}$ exist for each $[W_0]$ in (\ref{jh}), we only need to work out the case $X\in\mathcal{I}$ and $Y\in\mathcal{I}[-1]\cup\mathcal{P}$.

By Lemmas \ref{xuanzhuan} and \ref{homaut}, it suffices to show $\varphi_{L,X[-1]}^Y$ exists for $X\in\mathcal{I}$, $Y\in\mathcal{I}[-1]\cup\mathcal{P}$ and $L\in\mathcal{I}[-1]\cup\mod A$.

Assume that $L=L_1\oplus L_2$ with $L_1\in\mathcal{I}[-1]\cup\mathcal{P}\cup\mathcal{R}$ and $L_2\in\mathcal{I}$.
By the associativity for $\{L_1,L_2,X[-1];Y\}$, we obtain that
\begin{equation}\label{jh2}F_{L_1,L_2}^LF_{L,X[-1]}^Y=\sum\limits_{[U]}F_{L_1,U}^YF_{L_2,X[-1]}^U.\end{equation}
By the triangle $X[-1]\to U\to L_2\to X$, we get that $U\in\mathcal{I}[-1]\cup\mod A$, furthermore, by the triangle $L_1[-1]\to U\to Y\to L_1$, we conclude that $U\in\mathcal{I}[-1]\cup\mathcal{P}$. It follows that the sum in (\ref{jh2}) is over a finite set which is independent of the relative conservative field extensions. Note that $L_1, Y, U$ belong to the subcategory $\I[-1]\cup\P\cup \R$, which is equivalent to the category of coherent sheaves over a domestic weighted projective line. By \cite[Corollary 4.1]{DR}, the generic function $\varphi_{L_1,U}^{Y}$ exists. Moreover, for each $U$ in (\ref{jh2}), by Lemma \ref{xuanzhuan}, we have
$$F_{L_2,X[-1]}^U=F_{U,L_2[-1]}^{X[-1]}\cdot\frac{|\Aut U|\cdot\{U,U\}}{|\Aut X|\cdot\{X,X\}}.$$
Note that $U, L_2[-1], X[-1]\in\I[-1]\cup \P$, thus $\varphi_{U,L_2[-1]}^{X[-1]}$ and then $\varphi_{L_2,X[-1]}^U$ exists.
Therefore, we complete the proof.
\ep

\begin{proposition}\label{case of X mod}
The generic function $\varphi_{X,Y}^L$ exists for $X\in\mod A$ and $Y,L\in D^b(A)$.
\end{proposition}

\bp
For the case $X\in\mathcal{P}$ or $\mathcal{I}$, we are done by Lemma \ref{case of X non regular mod}. Hence, we assume that $X$ admits a decomposition $X=X_1\oplus X_2$ with $0\neq X_1\in\mathcal {P}\cup\mathcal {R}$ and $0\neq X_2\in\mathcal {R}\cup\mathcal {I}$. Then by the associativity for $\{X_1,X_2,Y;L\}$, we obtain that
$$\sum\limits_{[M]}F_{X_1,X_2}^MF_{M,Y}^L=\sum\limits_{[U]}F_{X_1,U}^LF_{X_2,Y}^U.$$
Hence, \begin{equation}\label{qiuhe1}F_{X_1,X_2}^XF_{X,Y}^L
=\sum\limits_{[U]}F_{X_1,U}^LF_{X_2,Y}^U-\sum\limits_{[M]\neq[X]}F_{X_1,X_2}^MF_{M,Y}^L.\end{equation}
By Lemma \ref{generic function for direct sum} $\varphi_{X_1,X_2}^X$ already exists, and by Lemma \ref{diyi} the first sum in (\ref{qiuhe1}) is independent of the relative conservative field extensions.
Moreover, since $\Ext^1(\mathcal {P},\mathcal {I})=\Ext^1(\mathcal {P},\mathcal {R})=\Ext^1(\mathcal {R},\mathcal {I})=0$ and different tubes are $\Ext$-orthogonal, we deduce from $F_{X_1,X_2}^M\neq 0$ that the second sum in (\ref{qiuhe1}) is also independent of the relative conservative field extensions.
Now using the similar method as in the proof of the Main Theorem for cyclic quiver case, we obtain the existence of the generic function $\varphi_{X,Y}^L$ by induction on $(l(X),d(X))$.
\ep

\textbf{\emph{Proof of the Main Theorem for $Q$ of tame type:}}

We write $X=\underset{0\leq i\leq n}{\mathop\bigoplus} X_i[i]$ with $X_i\in\mod A$ as above and prove by induction on $t(X):=n$.
For the case $t(X)=0$, we are done by Proposition \ref{case of X mod}. Now we suppose $t(X)>0$. We write $X_n=\tilde{X}_n\oplus \hat{X}_n$ with $\tilde{X}_n\in\P$ and $\hat{X}_n\in\R\cup\I$, and set $X^{(2)}=\hat{X}_n[n]$ and $X^{(1)}=X\backslash{X^{(2)}}$.
Then by the associativity for $\{X^{(1)},X^{(2)},Y;L\}$, we obtain that
\begin{equation}\label{jh3}F_{X^{(1)},X^{(2)}}^XF_{X,Y}^L=\sum\limits_{[U]}F_{X^{(1)},U}^LF_{X^{(2)},Y}^U.\end{equation}
By Lemma \ref{generic function for direct sum} $\varphi_{X^{(1)},X^{(2)}}^X$ already exists, and by Lemma \ref{diyi}, the sum on the right-hand side is independent of the relative conservative field extensions. Moreover, for each $[U]$ in (\ref{jh3}), $\varphi_{X^{(2)},Y}^U$ exists by Proposition \ref{case of X mod}.
Hence, it suffices to deal with $F_{X^{(1)},U}^L$. That is, we are reduced to prove the existence of the generic function $\varphi_{X,Y}^L$ for $X=\underset{0\leq i\leq n}{\mathop\bigoplus} X_i[i]$ with all $X_i\in\mod A$ and $X_n\in\P$.

Write $X_0=\tilde{X}_0\oplus \hat{X}_0$ with $\tilde{X}_0\in\P$ and $\hat{X}_0\in\R\cup\I$, and set $X^{(1)}=\tilde{X}_0$ and $X^{(2)}=X\backslash X^{(1)}$. Then the associativity of derived Hall algebras produces (\ref{jh3}) again. By Lemma \ref{diyi} the sum in (\ref{jh3}) is also independent of the relative conservative field extensions. Moreover, we have proved that both of $\varphi_{X^{(1)},X^{(2)}}^X$ and $\varphi_{X^{(1)},U}^L$ exist. Therefore, we are reduced to prove the existence of the generic function $\varphi_{X,Y}^L$ for $X=\underset{0\leq i\leq n}{\mathop\bigoplus} X_i[i]$ with all $X_i\in\mod A$, $X_0\in\R\cup\I$ and $X_n\in\P$. For this, we can choose a tilting object $T$ in $D^b(A)$ with endomorphism algebra $A'$, such that under the derived equivalence $D^b(A)\cong D^b(A')$, $X$ belongs to the subcategory $\cup_{0\leq i\leq n-1}\mod A'[i]$. Thus, by induction on $t(X)$, we conclude that the generic function $\varphi_{X,Y}^L$ exists. Therefore, we have finished the proof of Main Theorem.
\fin

%%%%%%%%%%%%%%%%%%%%%%%%%%%%%%%%%%%%%%%%%%%%%%%%%%%%%%%%%%%%%%%%%%%%%%%%%%%%%%%%%%%%%%%%%%%%%%%%%%%%%%%%%%%%%%%%%%%%%%%%%%%%
\section*{Acknowledgments}

The authors are grateful to Professor Bangming Deng for his helpful discussions and comments.

\end{document}